\documentclass[11pt,letterpaper]{amsart}
\usepackage{amsmath,mathtools,amssymb}
\usepackage{color}

\makeatletter
\@namedef{subjclassname@2020}{\textup{2020} Mathematics Subject Classification}
\makeatother

\mathtoolsset{showonlyrefs}

\usepackage{amsmath}
\usepackage{amsthm}
\usepackage{amsfonts}
\usepackage{amssymb}
\usepackage{mathrsfs}
\usepackage[shortlabels]{enumitem}
\usepackage{graphicx}
\usepackage{subfigure}
\usepackage{color}
\usepackage[dvipsnames]{xcolor}

\newtheorem{theorem}{Theorem}
\newtheorem{prop}[theorem]{Proposition}

\theoremstyle{definition}

\newtheorem{rem}[theorem]{Remark}

\newcommand{\N}{\mathbb{N}}
\newcommand{\Z}{\mathbb{Z}}

\newcommand{\R}{\mathbb{R}}

\newcommand\e{\mathrm{e}}
\newcommand\I{\mathrm{i}}
\newcommand\re{\operatorname{Re}}
\newcommand\im{\operatorname{Im}}

\newcommand\eps\varepsilon
\renewcommand\epsilon\varepsilon
\renewcommand\rho\varrho
\newcommand\lm\lambda

\newcommand{\dist}{\operatorname{dist}}

\newcommand{\rd}{\mathrm{d}}

\newcommand{\beq}{\begin{equation}}
\newcommand{\eeq}{\end{equation}}
\newcommand{\be}{\begin{equation*}}
\newcommand{\ee}{\end{equation*}}
\newcommand{\bmat}{\begin{pmatrix}}
\newcommand{\emat}{\end{pmatrix}}

\begin{document}
\title[Improved Lieb--Thirring type inequalities]{Improved Lieb--Thirring type inequalities for non-selfadjoint Schr\"odinger operators}
\author[S. Bögli]{Sabine Bögli}
\address[S. Bögli]{Department of Mathematical Sciences, Durham University, Upper Mountjoy Campus,  Durham DH1 3LE, United Kingdom}
\email{sabine.boegli@durham.ac.uk}

\date{\today}


\begin{abstract}
We improve the Lieb--Thirring type inequalities by Demuth, Hansmann and Katriel (J.\ Funct.\ Anal.\ 2009)  for Schr\"odinger operators with complex-valued potentials. Our result involves a positive, integrable function. We show that in the one-dimensional case the result is sharp in the sense that if we take a non-integrable function, then an analogous inequality cannot hold.
\end{abstract}

\maketitle

\section{Introduction}
Lieb--Thirring inequalities appeared first in the work of Lieb and Thirring in the proof of stability of matter, see~\cite{lie-thi_prl75,lie-thi_91}. Since then, the involved constants have been improved, and various attempts have been made to generalise the inequalities to  allow for non-selfadjoint Schr\"odinger operators. 

The classical Lieb--Thirring inequality for a Schr{\" o}dinger operator $-\Delta+V$ in $L^{2}(\R^{d})$ reads
\begin{equation}
\sum_{\lambda\in\sigma_{d}(-\Delta+V)}|\lambda|^{p-d/2}\leq C_{d,p}\|V\|_{L^{p}}^{p},
\label{eq:lieb-thirring_ineq_sa}
\end{equation}
for real-valued potentials  $V\in L^{p}(\R^{d})$, where the range for $p$ depends on the dimension~$d$ as follows:
\begin{align}
p\geq 1,& \quad \mbox{ if } d=1,\nonumber\\
p> 1,& \quad \mbox{ if } d=2,\label{eq:p_d_rel}\\
p\geq\tfrac{d}{2},& \quad \mbox{ if } d\geq 3.\nonumber
\end{align}
Here $\sigma_d(-\Delta+V)$ denotes the set of discrete eigenvalues, outside the essential spectrum $\sigma_e(-\Delta+V)=[0,\infty)$.
The inequality~\eqref{eq:lieb-thirring_ineq_sa} cannot be true for complex-valued $V\in L^{p}(\R^{d})$ with $p>(d+1)/2$ since then $\sigma_{d}(-\Delta+V)$ can accumulate anywhere in the essential spectrum, see~\cite{bog_cmp17,Boegli-Cuenin}. 

It was proved by Demuth, Hansmann and Katriel in \cite{DHK} that if $p\geq d/2+1$, then for any $\tau\in (0,1)$ we have the inequality
\begin{equation}\label{eq:DHK}
\sum_{\lambda\in\sigma_{d}(-\Delta+V)}\frac{\left(\dist(\lambda,[0,\infty))\right)^{p+\tau}}{|\lambda|^{d/2+\tau}}\leq C_{d,p,\tau}\|V\|_{L^{p}}^{p}
\end{equation}
for any (complex-valued) $V\in L^p(\R^d)$. Note that in the selfadjoint case ($V$ real-valued), the inequality reduces to \eqref{eq:lieb-thirring_ineq_sa} since all discrete eigenvalues are in $(-\infty,0)$ and hence ${\rm dist}(\lm,[0,\infty))=|\lm|$.
In \cite{dem-han-kat_ieop13}, it was published as an open problem to prove whether \eqref{eq:DHK} remains true for $\tau=0$, i.e.\ whether
\begin{equation}\label{eq:lieb-thirring_ineq_dist_form}
\sum_{\lambda\in\sigma_{d}(H)}\frac{\left(\dist(\lambda,[0,\infty))\right)^{p}}{|\lambda|^{d/2}}\leq C_{p,d}\|V\|_{L^{p}}^{p}.
\end{equation}
For $d=1$, a counterexample was found in \cite{Boegli-Stampach}; it is still an open problem whether the inequality can hold in higher dimensions $d\geq2$.

The main ingredient in the proof of \eqref{eq:DHK} is the following Lieb--Thirring type inequality by Frank, Laptev, Lieb and Seiringer  in~\cite{FLLS}, which sums only over all eigenvalues outside a sector:
If $p\geq d/2+1$ then, for any $t>0$,
\begin{equation}\label{eq:FLLS}
\sum_{\lambda\in\sigma_d(-\Delta+V),\atop |\im\lambda|\geq t \re\lambda} |\lambda|^{p-d/2} \leq C_{d,p}\left(1+\frac{2}{t}\right)^p\|V\|_{L^p}^p.
\end{equation}
Now \eqref{eq:DHK} follows by taking a weighted integral over the parameter $t$ of the sector. However, the weight of the integral wasn't chosen to optimise the inequality.
In this paper we improve the result by taking an optimal weight function.

In Theorem \ref{thm1} we prove a Lieb-Thirring type inequality where the left hand side of \eqref{eq:lieb-thirring_ineq_dist_form} is multiplied by $f\left(-\log\left({\rm dist}(\lm,[0,\infty))/|\lambda|\right)\right)$ where $f$ is a positive function that can decay quite slowly (slower than in the result \eqref{eq:DHK} which corresponds to $f(s)=\e^{-\tau s}$) but is still integrable, $\int_0^\infty f(s)\,\rd s<\infty$.
In Theorem \ref{thm2} we show that in dimension $d=1$ the inequality is sharp in the sense that if we take a non-integrable function, $\int_0^{\infty}f(s)\,\rd s=\infty$, then such an inequality cannot hold for all $V\in L^p(\R)$.
This suggests that the $t$-dependence (asymptically $t^{-p}$ as $t\to 0$) on the right hand side of \eqref{eq:FLLS} is optimal. We prove this sharpness in Theorem \ref{thm3}.
In Section \ref{sec:ex} we discuss a few classes of integrable functions $f$. Each class can be combined with Theorem~\ref{thm1} to give inequalities that are better than~\eqref{eq:DHK}.

We mention that in \cite{fra_tams18, Frank-Sabin, Hansmann, Laptev-Safronov}, different Lieb-Thirring type inequalities were proved. They don't reduce to \eqref{eq:lieb-thirring_ineq_sa} in the selfadjoint case, and are therefore difficult to compare with the equalities proved here.

\section{New Lieb--Thirring type inequalities}

In this section we prove new Lieb-Thirring type inequalities and discuss their sharpness.

\begin{theorem}\label{thm1}
Let $d\in\N$ and $p\geq d/2+1$.
Let $f:[0,\infty)\to (0,\infty)$ be a continuous, non-increasing function.
If $\int_0^\infty f(s)\,{\rm d}s<\infty$, then there exists $C_{d,p,f}>0$ such that, for any $V\in L^p(\R^d)$,
\begin{equation}\label{eq:thm1}
\begin{aligned}
&\sum_{\lambda\in\sigma_d(-\Delta+V)}\frac{{\rm dist}(\lambda,[0,\infty))^p}{|\lambda|^{d/2}}f\left(-\log\Big(\frac{{\rm dist}(\lm,[0,\infty))}{|\lambda|}\Big)\right)
 \leq C_{d,p,f}\|V\|_{L^p}^p.
\end{aligned}
\end{equation}
\end{theorem}

\begin{proof}
First we show that it suffices to prove the theorem for a continuous, non-increasing, integrable, piecewise $C^1$-function $f$ for which there exists $c>0$ with $(\log f)'\geq -c$ almost everywhere.
To this end, first note that since $f$ is continuous and non-increasing, we find a non-increasing $C^1$-function $f_1:[0,\infty)\to (0,\infty)$ with $f(s)\leq f_1(s)\leq 2 f(s)$ for all $s\in [0,\infty)$.
Let $c>0$. Assume that there exist $0\leq a_1<b_1\leq \infty$ such that $(\log f_1)'\geq -c$ on $[0,a_1]$ and $(\log f_1)'<-c$ on $(a_1,b_1)$. Then
$$f_1(s)=\exp\left(\log(f_1(a_1))+\int_{a_1}^s (\log f_1)'(t)\,\rd t\right)<f_1(a_1)\e^{-c(s-a_1)}, \quad s\in (a_1,b_1).$$
Let $b_1'\in (b_1,\infty)$ be the smallest point where the function $f_1$ intersects $s\mapsto f_1(a_1)\e^{-c(s-a_1)}$; set $b_1'=\infty$ if they don't intersect.
Let $f_2:[0,\infty)\to (0,\infty)$ be the continuous, non-increasing, piecewise $C^1$-function defined by 
$$f_2(s)=\begin{cases}f_1(s), &s\notin (a_1,b_1'), \\ f_1(a_1)\e^{-c(s-a_1)}, &s\in (a_1,b_1'). \end{cases}$$ 
If there exists one or more intervals in $[b_1',\infty)$ on which $(\log f_1)'<-c$, we repeat the procedure (inductively) and change the function analogously.
We then have finitely or infinitely many intervals $(a_n,b_n')$ (sorted by increasing $a_n$) such that the piecewise $C^1$-function $f_\infty:[0,\infty)\to (0,\infty)$  defined by 
$$f_\infty(s)=\begin{cases}f_1(s), &s\notin (a_n,b_n') \text{ for any $n$}, \\ f_1(a_n)\e^{-c(s-a_n)}, &s\in (a_n,b_n'), \end{cases}$$ 
is continuous, non-increasing, and satisfies $(\log f_\infty)'\geq -c$ almost everywhere.
Note that, using $f_1\leq 2f$ and $f_1(a_n)\e^{-c(b_n'-a_n)}=f_1(b_n')$ by construction of $b_n'$,
\begin{align*}
\int_0^\infty f_\infty(s)\,\rd s 
&\leq \int_0^\infty f_1(s)\,\rd s+\sum_n \int_{a_n}^{b_n'} f_1(a_n)\e^{-c(s-a_n)}\,\rd s\\
&\leq 2\int_0^\infty f(s)\,\rd s+\frac{1}{c}\sum_n f_1(a_n)(1-\e^{-c(b_n'-a_n)})\\
&= 2\int_0^\infty f(s)\,\rd s+\frac{1}{c}\sum_n (f_1(a_n)-f_1(b_n'))\\
&\leq 2\int_0^\infty f(s)\,\rd s+\frac{1}{c} f_1(a_1),
\end{align*}
where in the last inequality we used that $-f_1(b_n')+f_1(a_{n+1})\leq 0$ since $b_n'\leq a_{n+1}$ and $f_1$ is non-increasing.
This implies that since $f$ is integrable, so is $f_\infty$.
Now if we can show the theorem for $f_\infty$, then it also follows for $f$ since $f\leq f_1\leq f_{\infty}$.

In the following we write again $f$ instead of $f_\infty$.
To prove the theorem, we proceed in a similar way as in \cite[Corollary 3]{DHK}. The main difference is that we use a function $g$ such that $t\mapsto t^{-p}g(t)$ may go to infinity at $t=0$ faster than with the function $g(t)=t^{p-1+\tau}$ ($\tau\in (0,1)$) used in \cite[Corollary 3]{DHK} but  is still integrable on $(0,m)$ (for $m>0$ sufficiently small). 
We start with the estimate from \cite[Theorem 1 part 2]{FLLS}, see \eqref{eq:FLLS}.
We estimate the left hand side further by taking the sum only over $\lambda$ with $\re\lambda>0$ and ${\rm dist}(\lm,[0,\infty))=|\im\lm|\geq t |\lm|$.
Then, multiplying both sides by $g(t):=t^{p-1}f(-\log t)$ and integrating over $t\in (0,1)$, we get 
\begin{equation}
\label{eq:intineq}
\int_0^{1}\sum_{\lambda\in\sigma_d(-\Delta+V),\atop \re\lm>0,\,{\rm dist}(\lm,[0,\infty))\geq t |\lm|} |\lambda|^{p-d/2}g(t)\,\rd t
\leq C_{d,p} \|V\|_{L^p}^p \int_0^{1} \left(1+\frac{2}{t}\right)^p g(t)\,\rd t.
\end{equation}
Using that $1\leq 1/t$ for $t\in (0,1)$ and substituting $s=-\log t$, the integral on the right hand side of \eqref{eq:intineq} side becomes
\begin{align*}
\int_0^{1} \left(1+\frac{2}{t}\right)^p g(t)\,\rd t
\leq  \int_0^{1} 3^p t^{-1} f(-\log t)\,\rd t
= 3^p \int_{0}^\infty  f(s)\,\rd s<\infty.
\end{align*}
The left hand side of \eqref{eq:intineq} becomes, using the substitution $s=-\log t$,
\begin{equation}\label{eq:LHSint}
\begin{aligned}
&\int_0^{1}\sum_{\lambda\in\sigma_d(-\Delta+V),\atop  \re\lm>0,\,{\rm dist}(\lm,[0,\infty))\geq t |\lm|} |\lambda|^{p-d/2}g(t)\,\rd t\\
&=\sum_{\lambda\in\sigma_d(-\Delta+V),\atop\re\lambda>0}|\lambda|^{p-d/2}\int_0^{{\rm dist}(\lm,[0,\infty))/|\lm|}g(t)\,\rd t\\
&=\sum_{\lambda\in\sigma_d(-\Delta+V),\atop\re\lambda>0}|\lambda|^{p-d/2}\int_{-\log({\rm dist}(\lm,[0,\infty))/|\lm|)}^{\infty}\e^{-ps}f(s)\,\rd s.
\end{aligned}
\end{equation}
Recall that by assumption or construction, $f=f_\infty$ satisfies $(\log f)'\geq -c$ almost everywhere.
This implies $f'(s)\geq -c f(s)$ and thus integration by parts yields
$$\int_a^{\infty} \e^{-ps}f(s)\,\rd s\geq \frac{1}{p}\left(\e^{-pa}f(a)-c\int_a^{\infty} \e^{-ps}f(s)\,\rd s\right),$$
whence $$\int_a^{\infty} \e^{-ps}f(s)\,\rd s\geq \frac{1}{p+c} \e^{-pa}f(a).$$
Thus the last line in \eqref{eq:LHSint} can be estimated from below by
\begin{align*}
& \frac{1}{p+c} \sum_{\lambda\in\sigma_d(-\Delta+V),\atop\re\lambda>0}|\lambda|^{p-d/2}
 \e^{-pa}f(a)\Big|_{a=-\log\Big(\frac{{\rm dist}(\lm,[0,\infty))}{|\lambda|}\Big)}\\
&= \frac{1}{p+c} 
\sum_{\lambda\in\sigma_d(-\Delta+V),\atop \re\lm>0}|\lambda|^{p-d/2}
\Big(\frac{{\rm dist}(\lm,[0,\infty))}{|\lambda|}\Big)^p f\left(-\log\Big(\frac{{\rm dist}(\lm,[0,\infty))}{|\lambda|}\Big)\right).
\end{align*}

Combining the estimates of the left and right hand sides of \eqref{eq:intineq}, we obtain \eqref{eq:thm1} with the left hand side restricted to $\re\lm>0$.
For the sum over the eigenvalues with 
$\re\lm\leq 0$, we use \eqref{eq:FLLS} with $t=1$ (in fact, any $t>0$ would work here).
Note that then all $\lm$ satisfying $\re\lm\leq 0$ also satisfy $|\im\lm|\geq t\re\lm$.
Thus we get
$$\sum_{\lambda\in\sigma_d(-\Delta+V),\atop \re\lm\leq 0} |\lambda|^{p-d/2} \leq C_{d,p} 3^p\|V\|_{L^p}^p.$$
This implies, with ${\rm dist}(\lm,[0,\infty))=|\lm|$ if $\re\lm\leq 0$,
\begin{align*}
&\sum_{\lambda\in\sigma_d(-\Delta+V),\atop \re\lm\leq 0} \frac{{\rm dist}(\lm,[0,\infty))^p}{|\lambda|^{d/2}}
f\left(-\log\Big(\frac{{\rm dist}(\lm,[0,\infty))}{|\lambda|}\Big)\right)\\
&= \sum_{\lambda\in\sigma_d(-\Delta+V),\atop \re\lm\leq 0} |\lambda|^{p-d/2}
f\left(0\right)
 \leq C_{d,p}3^p f\left(0\right) \|V\|_{L^p}^p,
\end{align*}
which concludes the proof.
\end{proof}

In dimension $d=1$, Theorem~\ref{thm1} is optimal in the sense that if $f$ is no longer integrable, then the corresponding Lieb--Thirring inequality is false.
\begin{theorem}\label{thm2}
Let $d=1$ and $p\geq 1$.
Let $f:[0,\infty)\to (0,\infty)$ be a continuous, non-increasing function with $\int_0^\infty f(s)\,{\rm d}s=\infty$.
Then
\begin{equation}\label{eq:thm2}
\begin{aligned}
&\sup_{V\in L^p(\R)}\,\frac{\sum_{\lambda\in\sigma_d(-\Delta+V)}\frac{{\rm dist}(\lambda,[0,\infty))^p}{|\lambda|^{1/2}}f\left(-\log\Big(\frac{{\rm dist}(\lm,[0,\infty))}{|\lambda|}\Big)\right)}
{\|V\|_{L^p}^p}=\infty.
\end{aligned}
\end{equation}
\end{theorem}

\begin{rem}
Note that Theorem \ref{thm2} holds for $p\geq 1$ but Theorem \ref{thm1} is proved for $p\geq 3/2$ in dimension $d=1$. It is still an open question how to optimise the Lieb-Thirring type inequalities in $d=1$ for $p\in (1,3/2)$, and in higher dimensions. In particular, it is not known whether Theorem \ref{thm1} is optimal for $d\geq 2$.
\end{rem}

\begin{proof}[Proof of Theorem {\rm \ref{thm2}}]
We use a slight modification of the example in \cite[Proposition 10]{Boegli-Stampach} with potential $V_h(x)=\I h \chi_{[-1,1]}(x)$ where $h>0$ with $h\to\infty$. In the proof of  \cite[Proposition 10]{Boegli-Stampach}, asymptotic formulas were derived for all eigenvalues $\lambda_j=\mu_j^2+\I h$ with $\alpha \log h\leq \im \mu_j\leq \alpha \log h$, $\re\mu_j\leq -h^\gamma |\im \mu|$ where $\alpha,\beta,\gamma>0$ are chosen to be $h$-independent constants with $\gamma<2\alpha<2\beta<1$.
We now want to allow $\alpha,\gamma$ to be $h$-dependent ($\beta\in (0,1/2)$ still constant), with $0<\gamma(h)<2\alpha(h)\to 0$ as $h\to \infty$, but sufficiently slowly so that the asymptotics still hold. It turns out that if we still have $h^{-\gamma(h)}\to 0$, then
$$\mu_j=\frac{\pi}{4}(7-8j)+\I\log\left(\frac{\pi(8j-7)}{2\sqrt{h}}\right)+O(h^{-\gamma(h)})$$
uniformly for all integers $j$ with $h^{\alpha(h)+1/2}\leq j\leq h^{\beta+1/2}$.
Hence we set $\eps(h)=h^{-\gamma(h)}$ (converging to zero arbitrarily slowly) and $\alpha(h)=\gamma(h)$. The latter implies $h^{\alpha(h)+1/2}=\frac{h^{1/2}}{\eps(h)}$.
Analogously we then obtain that there exists $h_0>0$ such that for all $h>h_0$ and $\frac{h^{1/2}}{\eps(h)}\leq j\leq h^{\beta+1/2}$, we have the estimates
\begin{equation}\label{eq:lambdaj}
{\rm dist}(\lambda_j,[0,\infty))=\im \lambda_j>\frac{h}{2}, \quad \pi j\leq  |\lambda_j|^{1/2}\leq 2\pi j.
\end{equation}
This implies, for all $h>h_0$,
\begin{align*}
& \sum_{\lambda\in\sigma_d(-\Delta+V_h)}\frac{{\rm dist}(\lambda,[0,\infty))^p}{|\lambda|^{1/2}}f\left(-\log\Big(\frac{{\rm dist}(\lm,[0,\infty))}{|\lambda|}\Big)\right)\\
&\geq \sum_{\frac{h^{1/2}}{\eps(h)}\leq j\leq h^{\beta+1/2}}\frac{{\rm dist}(\lambda_j,[0,\infty))^p}{|\lambda_j|^{1/2}}f\left(-\log\Big(\frac{{\rm dist}(\lm_j,[0,\infty))}{|\lambda_j|}\Big)\right)\\
&\geq \left(\frac{h}{2}\right)^p \frac{1}{2\pi} 
\sum_{\frac{h^{1/2}}{\eps(h)}\leq j\leq h^{\beta+1/2}}\frac{1}{j} f\left(-\log\Big(\frac{h}{2(2\pi j)^2}\Big)\right).
\end{align*}
Note that $-\log\Big(\frac{h}{2(2\pi j)^2}\Big)=\log\left(\frac{j^2}{h}\right)+\log(8\pi^2)$.
Since $f$ is non-increasing, with $f(s)\leq f(0)$ for $s\geq 0$, we see that
\begin{align*}
&\sum_{\frac{h^{1/2}}{\eps(h)}\leq j\leq h^{\beta+1/2}}\frac{1}{j} f\left(\log\left(\frac{j^2}{h}\right)+\log(8\pi^2)\right)\\
&\geq \int_{\frac{h^{1/2}}{\eps(h)}}^{h^{\beta+1/2}} \frac{1}{x} f\left(\log\left(\frac{x^2}{h}\right)+\log(8\pi^2)\right)\,\rd x-2f(0)\\
&\geq \int_{-2\log\left(\eps(h)\right)}^{2\beta\log(h)} f(s)\,\rd s -2f(0)
\end{align*}
where we have used the substitution $s=\log\left(\frac{x^2}{h}\right)$. Note that $2\beta\log(h)\to \infty$.
Since $f$ is not integrable, we can choose $\eps(h)\to 0$ so slowly (i.e.\ $-2\log(\eps(h))\to\infty$ so slowly) that
$\int_{-2\log\left(\eps(h)\right)}^{2\beta\log(h)} f(s)\,\rd s\to \infty$.
Now, with $\|V_h\|_{L^p}^p=2h^p$, we arrive at
 \begin{align*}
&\frac{\sum_{\lambda\in\sigma_d(-\Delta+V_h)}\frac{{\rm dist}(\lambda,[0,\infty))^p}{|\lambda|^{1/2}}f\left(-\log\Big(\frac{{\rm dist}(\lm,[0,\infty))}{|\lambda|}\Big)\right)}{\|V_h\|_{L^p}^p}\\
&\geq \left(\frac{1}{2}\right)^p \frac{1}{4\pi} \left(\int_{-2\log\left(\eps(h)\right)}^{2\beta\log(h)} f(s)\,\rd s -2f(0)\right)\to \infty, \quad h\to\infty.
\end{align*}
This completes the proof.
\end{proof}

Finally we also prove that in one dimension, the $t$-dependence in \eqref{eq:FLLS} is optimal.
\begin{theorem}\label{thm3}
Let $d=1$ and $p\geq 1$.
Let $\varphi:(0,\infty)\to (0,\infty)$ a continuous function with $\varphi(t)=o(t^{-p})$ as $ t\to 0$.
Then
\begin{equation}\label{eq:thm3}
\begin{aligned}
&\limsup_{t\to 0}\sup_{V\in L^p(\R)}\,\frac{\sum\limits_{\lambda\in\sigma_d(-\Delta+V),\atop |\im\lambda|\geq t \re\lambda} |\lambda|^{p-1/2}}
{\varphi(t)\|V\|_{L^p}^p}=\infty.
\end{aligned}
\end{equation}
\end{theorem}

\begin{proof}
We use the same potentials $V_h$ as in the proof of Theorem \ref{thm2}. Recall that $\|V_h\|_{L^p}^p=2h^p$.
Again we use the index $j\in \Z$ with $\frac{h^{1/2}}{\eps(h)}\leq j\leq h^{\beta+1/2}$ to enumerate the eigenvalues $\lm_j$ with uniform estimates in \eqref{eq:lambdaj}; we take $\eps(h)\to 0$ with $1/\eps(h)=o(h^\beta)$.
We take $t$ to be $h$-dependent as $t(h)=\frac{h^{-2\beta}}{8\pi^2}$. Then one can check that each of these $\lm_j$ satisfies
$|\im\lm_j|\geq t(h) |\lm_j|\geq t(h)\re\lm_j$.
Therefore,
\begin{equation}\label{eq:quotient}
\frac{\sum\limits_{\lambda\in\sigma_d(-\Delta+V_h),\atop |\im\lambda|\geq t(h) \re\lambda} |\lambda|^{p-1/2}}
{\varphi(t(h))\|V_h\|_{L^p}^p}
\geq \frac{\sum\limits_{\frac{h^{1/2}}{\eps(h)}\leq j\leq h^{\beta+1/2}} (\pi j)^{2(p-1/2)}}
{\varphi(t(h))2h^p}.
\end{equation}
In the limit $h\to\infty$, the sum on the right hand side is of the same order as 
$$\int_{\frac{h^{1/2}}{\eps(h)}}^{h^{\beta+1/2}}j^{2p-1}\,\rd j= \frac{h^{2\beta p+p}-\frac{h^{p}}{\eps(h)^{2p}}}{2p}.$$
Thus, up to a multiplicative constant, the right hand side of \eqref{eq:quotient} is asymptotically
$$\frac{h^{2\beta p}}{\varphi(t(h)}=\frac{1}{(8\pi^2)^p}\,\frac{t(h)^{-p}}{\varphi(t(h))}\to \infty,$$
where we used the assumption $\varphi(t)=o(t^{-p})$ as $t\to 0$. This proves the claim.
\end{proof}

\section{Examples}\label{sec:ex}
In this section we verify the assumptions of Theorem \ref{thm1} for a few examples.
If we combine the below part (i) with Theorem \ref{thm1}, we recover \cite[Corollary 3]{DHK}, which was the hitherto best known result. Parts (ii)--(v) yield improvements of this result.

We use the notation $g^{\circ n}$ for the \emph{$n$-th iterated} function $g$, i.e.\ $g^{\circ 0}(s)=s$, $g^{\circ 1}(s)=g(s)$, $g^{\circ 2}(s)=g(g(s))$ etc.
We also use the \emph{super-logarithm} (to the basis $\e$) defined for an $s\in \R$ by ${\rm slog}(s):=\min\{n\in\N_0:\,\log^{\circ n}(s)\leq 1\}$.

\begin{prop}\label{prop:ex}
The following continuous, non-increasing functions  $f:[0,\infty)\to (0,\infty)$ satisfy $\int_0^\infty f(s)\,{\rm d}s<\infty$:
\begin{enumerate}
\item[\rm (i)] $f(s)=\e^{-\eps s}$;
\item [\rm (ii)] $f(s)=\frac{1}{s^{1+\eps}}$ for $\eps>0$;
\item[\rm (iii)] $f(s)=\begin{cases} 1/\e & s\leq \e,\\ \frac{1}{s} \frac{1}{\log(s)^{1+\eps}}, &s>\e,\end{cases}$ for $\eps>0$;
\item[\rm (iv)] $f(s)=\begin{cases}1/\exp^{\circ n}(1), &{\rm slog}(s)\leq n \text{\rm\, (i.e.\ $s\leq \exp^{\circ n}(1)$)},\\
\left(\prod_{j=0}^{n-1}\frac{1}{\log^{\circ j}(s)}\right)\frac{1}{\log^{\circ n}(s)^{1+\eps}}, &{\rm slog}(s)\geq n+1  \text{\rm\, (i.e.\ $s> \exp^{\circ n}(1)$)},\end{cases}$ for $n\in\N$ and $\eps>0$ {\rm(}where $n=1$ corresponds to the case {\rm (iii))};
\item[\rm (v)] $f(s)=\begin{cases} 1/\e, &s\leq \e,\\ 
\left(\prod_{j=0}^{{\rm slog}(s)-1}\frac{1}{\log^{\circ j}(s)}\right)\frac{1}{({\rm slog}(s)-1+\log^{\circ{\rm slog}(s)}(s))^{1+\eps}}, &s>\e,\end{cases}$ for $\eps>~0$.
\end{enumerate}
In each case {\rm(i)--(v)}, if we set $\eps=0$ then $\int_0^{\infty}f(s)\,\rd s=\infty$.
\end{prop}

\begin{rem}\label{remthm2}
The functions in {\rm(i)--(iv)} are the piecewise derivatives of the following respective functions:
\begin{enumerate}
\item[\rm (i)] $F(s)=-\frac{1}{\eps}\e^{-\eps s}$;
\item[\rm (ii)] $F(s)=-\frac{1}{\eps}\frac{1}{s^\eps}$;
\item[\rm (iii)] $F(s)=-\frac{1}{\eps}\frac{1}{\log(s)^\eps}$ for $s>\e$;
\item[\rm (iv)] $F(s)=-\frac{1}{\eps}\frac{1}{(\log^{\circ n}(s))^\eps}$ for $s>\exp^{\circ n}(1)$;
\item[\rm (v)] $F(s)=-\frac{1}{\eps}\frac{1}{({\rm slog}(s)-1+\log^{\circ{\rm slog}(s)}(s))^{\eps}}$ for $s>\e$.
\end{enumerate}
In {\rm (v)},  note that $\log^{\circ{\rm slog}(s)}(s)\in (0,1]$, so asymptotically we have $F(s)\sim -\frac{1}{\eps}\frac{1}{({\rm slog}(s))^\eps}$ as $s\to\infty$.
More examples can be generated by taking $F$ to be a piecewise continuous function with $F(s)\nearrow 0$ very slowly as $s\to\infty$,  and then take $f=F'$ piecewise.
There exists no function with the slowest convergence, as if $F(s)=-1/g(s)$ with $g(s)\nearrow \infty$ slowly, then $g(g(s))\nearrow\infty$ even slower and hence $-1/g(g(s))\nearrow 0$ even slower than $F(s)$.
\end{rem}

\begin{proof}[Proof of Proposition{\rm~\ref{prop:ex}}]
By Remark \ref{remthm2}, we can write $\int_0^{\infty}f(s)\,\rd s=\int_0^a  f(s)\,\rd s+ F(a)$ for any $a\in (0,\infty)$. This proves that $f$ is integrable.
It is left to show that if $\eps=0$, then $\int_0^\infty f(s)\,\rd s=\infty$.
The corresponding antiderivatives (modulo additive constants) are the following functions:
\begin{enumerate}
\item[\rm (i)] $F(s)=s$;
\item[\rm (ii)] $F(s)=\log(s)$;
\item[\rm (iii)] $F(s)=\log^{\circ 2}(s)$ for $s>\e$;
\item[\rm (iv)] $F(s)=\log^{\circ (n+1)}(s)$ for $s>\exp^{\circ n}(1)$;
\item[\rm (v)] $F(s)=\log({\rm slog}(s)-1+\log^{\circ{\rm slog}(s)}(s))$ for $s>\e$.
\end{enumerate}
Since each of these functions satisfies $\lim_{s\to\infty}F(s)=\infty$, the corresponding $f=F'$ is not integrable.
\end{proof}

\section*{Acknowledgements}
The author thanks Jean-Claude Cuenin, Rupert L.\ Frank and Franti\v sek \v Stampach for helpful discussions.

\bibliography{mybib}{}
\bibliographystyle{acm}

\end{document}